\documentclass[a4paper,7v pt]{article}
\usepackage{amsfonts}
\usepackage{amssymb,amsthm}
\usepackage{amsmath,comment}
\usepackage{graphicx,graphics,color}

\newcommand{\R}{\mathbb{R}}
\newcommand{\N}{\mathbb{N}}

\newcommand{\Pa}{\mathbb{P}}

\newtheorem{theorem}{Theorem}
\newtheorem{lemma}[theorem]{Lemma}

\def \Leb {{\mathcal{L}}}

\begin{document}
\title{Spectral bounds for the torsion function}
\author{{M. van den Berg} \\
%EndAName
School of Mathematics, University of Bristol\\
University Walk, Bristol BS8 1TW\\
United Kingdom\\
\texttt{mamvdb@bristol.ac.uk}\\}
\date{30 March 2017}\maketitle
\vskip 3truecm \indent

\begin{abstract}\noindent Let $\Omega$ be an open set in Euclidean space $\R^m,\, m=2,3,...$, and let $v_{\Omega}$ denote the torsion function for $\Omega$. It is known that $v_{\Omega}$ is bounded if and only if the bottom of the spectrum of the Dirichlet Laplacian acting in $\Leb^2(\Omega)$, denoted by $\lambda(\Omega)$, is bounded away from $0$. It is shown that the previously obtained bound $\|v_{\Omega}\|_{\Leb^{\infty}(\Omega)}\lambda(\Omega)\ge 1$ is sharp: for $m\in\{2,3,...\}$, and any $\epsilon>0$ we construct an open, bounded and connected set $\Omega_{\epsilon}\subset \R^m$ such that $\|v_{\Omega_{\epsilon}}\|_{\Leb^{\infty}(\Omega_{\epsilon})} \lambda(\Omega_{\epsilon})<1+\epsilon$.
An upper bound for $v_{\Omega}$ is obtained for planar, convex sets in Euclidean space $M=\R^2$, which is sharp in the limit of elongation.
For a complete, non-compact, $m$-dimensional Riemannian manifold $M$ with non-negative Ricci curvature, and without boundary it is shown that $v_{\Omega}$ is bounded if and only if the bottom of the spectrum of the Dirichlet-Laplace-Beltrami operator acting in $\Leb^2(\Omega)$ is bounded away from $0$.
\end{abstract}

\noindent
{\it AMS} 2000 {\it subject classifications.} 58J32; 58J35; 35K20.\\

\medskip\noindent
{\it Acknowledgement.} MvdB acknowledges support by The Leverhulme Trust
through International Network Grant \emph{Laplacians, Random Walks, Bose Gas,
Quantum Spin Systems}.

%\mbox{}\newpage

\section{Introduction\label{sec1}}
Let $\Omega$ be an open set in $\R^m,$ and let $\Delta$ be the
Laplace operator acting in $L^2(\R^m)$. Let $(B(s),s\ge 0, \Pa_x,x\in \R^m)$ be Brownian motion
on $\R^m$ with generator $\Delta$. For $x\in \Omega$ we denote the first exit time, and expected lifetime of Brownian motion by

\begin{equation*}%\label{e1}
T_{\Omega}=\inf\{s\ge 0: B(s)\notin \Omega\},
\end{equation*}
and
\begin{equation}\label{e2}
v_{\Omega}(x)=\mathbb{E}_x[T_{\Omega}],\, x\in \Omega,
\end{equation}
respectively, where $\mathbb{E}_x$ denotes the expectation associated with $\mathbb{P}_x$. Then
$v_{\Omega}$ is the torsion function for $\Omega$, i.e. the unique solution of
\begin{equation}\label{e3}
-\Delta v=1,\, v\in H_0^1(\Omega).
\end{equation}
The bottom of the spectrum of the Dirichlet Laplacian acting in $\Leb^2(\Omega)$ is denoted by
\begin{equation}\label{e4}
\lambda(\Omega)=\inf_{\varphi\in H_0^1(\Omega)\setminus\{0\}}\frac{\displaystyle\int_\Omega|D\varphi|^2}{\displaystyle\int_\Omega \varphi^2}.
\end{equation}
It was shown in \cite{vdB}, \cite{vdBC} that $\|v_{\Omega}\|_{\Leb^{\infty}(\Omega)}$ is finite if and only if  $\lambda(\Omega)>0$. Moreover, if $\lambda(\Omega)>0$, then
\begin{equation}\label{e5}
\lambda(\Omega)^{-1}\le \|v_{\Omega}\|_{\Leb^{\infty}(\Omega)}\le (4+3m\log 2)\lambda(\Omega)^{-1}.
\end{equation}
The upper bound in \eqref{e5} was subsequently improved (see \cite{HV}) to
\begin{equation*}%\label{e6}
\|v_{\Omega}\|_{\Leb^{\infty}(\Omega)}\le \frac18(m+cm^{1/2}+8)\lambda(\Omega)^{-1},
\end{equation*}
where
\begin{equation*}%\label{e7}
c=(5(4+\log 2))^{1/2}.
\end{equation*}

In Theorem \ref{the3} below we show that the coefficient $1$ of $\lambda(\Omega)^{-1}$ in the left-hand side of \eqref{e5} is sharp.
\begin{theorem}\label{the3}
For $m\in\{2,3,\dots\}$, and any $\epsilon>0$ there exists an open, bounded, and connected set $\Omega_{\epsilon}\subset\R^m$ such that
\begin{equation}\label{e11}
\|v_{\Omega_{\epsilon}}\|_{\Leb^{\infty}(\Omega_{\epsilon})} \lambda(\Omega_{\epsilon})<1+\epsilon.
\end{equation}
\end{theorem}
\noindent The set $\Omega_{\epsilon}$ is constructed explicitly in the proof of Theorem \ref{the3}.

It has been shown by L. E. Payne (see (3.12) in \cite{P}) that for any convex, open $\Omega\subset\R^m$ for which $\lambda(\Omega)>0$,
\begin{equation}\label{e9}
\|v_{\Omega}\|_{\Leb^{\infty}(\Omega)}\lambda(\Omega)\ge \frac{\pi^2}{8},
\end{equation}
with equality if $\Omega$ is a slab, i.e. the connected, open set, bounded by two parallel $(m-1)$-dimensional hyperplanes.
Theorem \ref{the2} below shows that for any sufficiently elongated, convex, planar set (not just an elongated rectangle) $\|v_{\Omega}\|_{\Leb^{\infty}(\Omega)}\lambda(\Omega)$ is approximately equal to $\frac{\pi^2}{8}$. We denote the width and the diameter of a bounded open set $\Omega$ by $w(\Omega)$ (i.e. the minimal distance of two parallel lines supporting $\Omega$), and $\textup{diam}(\Omega)=\sup\{|x-y|:x\in \Omega,\,y\in \Omega\}$ respectively.
\begin{theorem}\label{the2}If $\Omega$ is a bounded, planar, open, convex set with width $w(\Omega)$, and diameter $\textup{diam}(\Omega)$,
then
\begin{equation*}%\label{e10}
\|v_{\Omega}\|_{\Leb^{\infty}(\Omega)}\lambda(\Omega)\le \frac{\pi^2}{8}\left(1+7\cdot3^{2/3}\left(\frac{w(\Omega)}{\textup{diam}(\Omega)}\right)^{2/3}\right).
\end{equation*}
\end{theorem}

 In the Riemannian manifold setting we denote the bottom of the spectrum of the Dirichlet-Laplace-Beltrami operator by \eqref{e4}. We have the following.
\begin{theorem}\label{the1}Let $M$ be a complete, non-compact, $m$-dimensional Riemannian manifold, without boundary, and with non-negative Ricci curvature. There exists $K<\infty,$ depending on $M$ only, such that if $\Omega\subset M$ is open, and $\lambda(\Omega)>0,$ then
\begin{equation}\label{e8}
 \lambda(\Omega)^{-1}\le \|v_{\Omega}\|_{\Leb^{\infty}(\Omega)}\le 2^{(3m+8)/4}\cdot3^{m/2}K^2 \lambda(\Omega)^{-1},
 \end{equation}
 where $K$ is the constant in the Li-Yau inequality in \eqref{e21} below.
 \end{theorem}

The proofs of Theorems \ref{the3}, \ref{the2}, and \ref{the1} will be given in Sections \ref{sec4}, \ref{sec3} and \ref{sec2} respectively.

Below we recall some basic facts on the connection between torsion function and heat kernel. It is well known (see \cite{EBD3}, \cite{GB}, \cite{GB1}) that the heat equation
\begin{equation*}%\label{e12}
\Delta u(x;t)=\frac{\partial u(x;t)}{\partial t},\quad x\in M,\quad t>0, \
\end{equation*}
has a unique, minimal, positive fundamental solution $p_M(x,y;t),$
where $x\in M$, $y\in M$, $t>0$. This solution, the
heat kernel for $M$, is symmetric in $x,y$, strictly positive,
jointly smooth in $x,y\in M$ and $t>0$, and it satisfies the
semigroup property
\begin{equation*}%\label{e13}
p_M(x,y;s+t)=\int_{M}dz\ p_M(x,z;s)p_M(z,y;t),
\end{equation*}
for all $x,y\in M$ and $t,s>0$, where $dz $ is the Riemannian
measure on $M$. See, for example, \cite{RS} for details. If $\Omega$ is an open subset of $M,$ then we denote the unique, minimal, positive fundamental solution of the heat equation on $\Omega$ by $p_{\Omega}(x,y;t)$, where $x\in \Omega,y\in \Omega,t>0$. This Dirichlet heat kernel satisfies,
\begin{equation*}%\label{e14}
p_{\Omega}(x,y;t)\le p_M(x,y;t),\, x\in \Omega, y\in \Omega,t>0.
\end{equation*}
Define $u_{\Omega}:\Omega \times (0,\infty)\mapsto \R$ by
\begin{equation*}%\label{e15}
    u_{\Omega}(x;t)=\int_{\Omega}dy\, p_{\Omega}(x,y;t).
\end{equation*}
Then,
\begin{equation*}%\label{e16}
    u_{\Omega}(x;t)=\int_{\Omega}dy\, p_{\Omega}(x,y;t)=\Pa_x[T_{\Omega}>t],
\end{equation*}
and by \eqref{e2}
\begin{equation}\label{e17}
    v_{\Omega}(x)=\int_0^{\infty}dt\,\Pa_x[T_{\Omega}>t]=\int_0^{\infty}dt\,\int_{\Omega}dy\, p_{\Omega}(x,y;t).
\end{equation}
It is straightforward to verify that $v_{\Omega}$ as in \eqref{e17} satisfies \eqref{e3}.

\section{Proof of Theorem \ref{the3}\label{sec4}}

We introduce the following notation. Let $C_L=(-\frac{L}{2},\frac{L}{2})^{m/2}$ be the open cube with measure $L^m$, and delete from $C_L$, $N^m$ closed balls with radii $\delta$,
where each ball $B(c_i;\delta)$ is positioned at the centre of an open cube $Q_i$ with measure $(L/N)^m$. These open cubes are pairwise disjoint, and contained in $C_L$.
Let $0<\delta<\frac{L}{2N}$, and put
\begin{equation*}\label{e41}
\Omega_{\delta,N,L}=C_L-\cup_{i=1}^{N^m}B(c_i;\delta).
\end{equation*}
Below we will show that for any $\epsilon>0$ we can choose $\delta, N$ such that
\begin{equation*}%\label{e42}
\|v_{\Omega_{\delta,N,L}}\|_{\Leb^{\infty}(\Omega_{\delta,N,L})} \lambda(\Omega_{\delta,N,L})<1+\epsilon.
\end{equation*}\begin{figure}[h]
\centering\includegraphics[scale=.6]{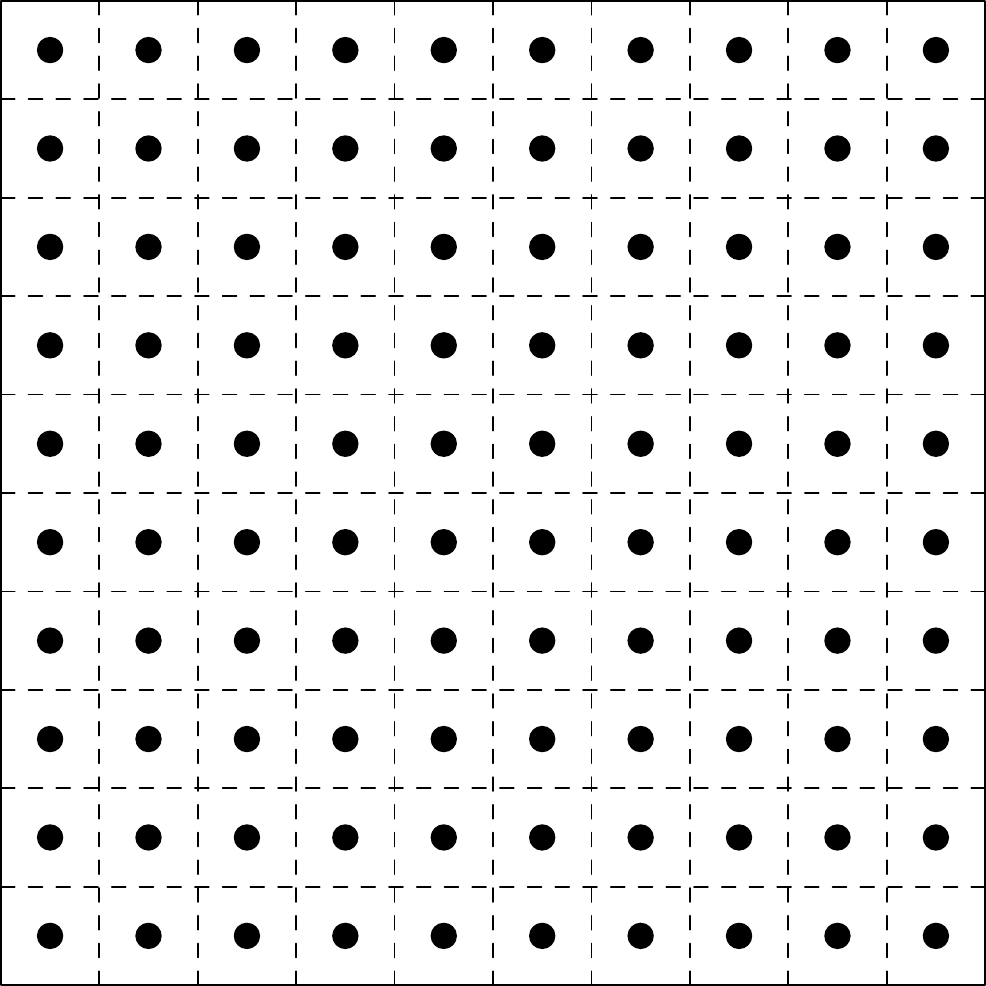}
\caption{$\Omega_{\delta,N,L}$ with $m=2,N=10,\delta=\frac{L}{8N}.$}\label{fig1}
\end{figure}

In Lemma \ref{lem2} below we show that $\lambda(\Omega_{\delta,N,L})$ is approximately equal to the first eigenvalue, $\mu_{1,B(0;\delta),L/N},$ of the Laplacian with Neumann boundary conditions on $\partial C_{L/N}$, and with Dirichlet boundary conditions on $\partial B(0;\delta)$.
The requirement $\mu_{1,B(0;\delta),L/N}$ not being too small stems from the fact that the approximation of replacing the Neumann boundary conditions on $C_L$ is a surface effect which should not dominate
the leading term $\mu_{1,B(0;\delta),L/N}$.
\begin{lemma}\label{lem2}
If $\delta\le \frac{L}{4N},\ N\ge 10$, and $\frac{N}{L^2}\le \mu_{1,B(0;\delta),L/N}$, then
\begin{equation*}%\label{e43}
\lambda(\Omega_{\delta,N,L})\le \mu_{1,B(0;\delta),L/N}+32m\bigg(\frac{5}{4}\bigg)^m\bigg(\frac{N}{L^2}+\frac{1}{N^{1/2}}\mu_{1,B(0;\delta),L/N}\bigg)
\end{equation*}
\end{lemma}
\begin{proof}
Let $\varphi_{1,B(0;\delta),L/N}$ be the first eigenfunction (positive) corresponding to $\mu_{1,B(0;\delta),L/N}$, and normalised in $\Leb^2(C_{L/N}-B(0;\delta))$.
In order to prove the lemma we construct a test function by periodically extending $\varphi_{1,B(0;\delta),L/N}$ to all cubes $Q_1,\dots Q_{N^m}$ of $\Omega_{\delta,N,L}$. We denote this periodic extension by $f$. We define
\begin{equation*}%\label{e44}
C_{L,N}=C_{L(1-\frac{2}{N})}.
\end{equation*}
So $C_{L,N}$ is the sub-cube of $C_L$ with the outer layer of cubes of size $L/N$ removed. Let
\begin{equation*}%\label{e45}
\tilde{f}=\bigg(1-\frac{\textup{dist}(x,C_{L,N})}{L/(4N)}\bigg)_+f.
\end{equation*}
Then $\tilde{f}\in H_0^1(\Omega_{\delta,N,L}),$ and
\begin{equation}\label{e46}
\|\tilde{f}\|_{\Leb^2(\Omega_{\delta,N,L})}\ge \int_{C_{L,N}}f^2=(N-2)^m,
\end{equation}
since $f$ restricted to any of the cubes $Q_i$ in $\Omega_{\delta,N,L}$ is normalised. Furthermore
\begin{align*}%\label{e47}
|D\tilde{f}|^2&\le \bigg(1-\frac{\textup{dist}(x,C_{L,N})}{L/(4N)}\bigg)^2|Df|^2+1_{C_L-C_{L,N}}\bigg(\bigg(\frac{4N}{L}\bigg)^2f^2+\frac{8N}{L}f|Df|\bigg)\nonumber \\ &
\le |Df|^2+\bigg(\frac{4N}{L}\bigg)^21_{C_L-C_{L,N}}f^2+\frac{8N}{L}1_{C_L-C_{L,N}}f|Df|.
\end{align*}
Hence
\begin{align}\label{e48}
&\int_{\Omega_{\delta,N,L}}|D\tilde{f}|^2\le \int_{\Omega_{\delta,N,L}}|Df|^2+\bigg(\frac{4N}{L}\bigg)^2\int_{C_L-C_{L,N}}f^2\nonumber \\ & \hspace{4cm} +\frac{8N}{L}\bigg(\int_{C_L-C_{L,N}}|Df|^2\bigg)^{1/2}\bigg(\int_{C_L-C_{L,N}}f^2\bigg)^{1/2}\nonumber \\ &
=N^m\mu_{1,B(0;\delta),L/N}+\big(N^m-(N-2)^m\big)\bigg(\bigg(\frac{4N}{L}\bigg)^2+\frac{8N}{L}\big(\mu_{1,B(0;\delta),L/N}\big)^{1/2}\bigg)\nonumber \\ &
\le N^m\mu_{1,B(0;\delta),L/N}+\big(N^m-(N-2)^m\big)\bigg(\bigg(\frac{4N}{L}\bigg)^2+8N^{1/2}\mu_{1,B(0;\delta),L/N}\bigg),
\end{align}
where we have used the last hypothesis in the lemma.
By \eqref{e46}, \eqref{e48}, the Rayleigh-Ritz variational formula, and the hypothesis $N\ge 10$,
\begin{align}\label{e49}
\lambda(\Omega_{\delta,N,L})&\le \mu_{1,B(0;\delta),L/N}\nonumber \\ & \ \ \ +\frac{N^m-(N-2)^m}{(N-2)^m}\bigg(\bigg(\frac{4N}{L}\bigg)^2+\big(8N^{1/2}+1\big)\mu_{1,B(0;\delta),L/N}\bigg)\nonumber \\ &\le
\mu_{1,B(0;\delta),L/N}+32m\bigg(\frac{5}{4}\bigg)^m\bigg(\frac{N}{L^2}+\frac{1}{N^{1/2}}\mu_{1,B(0;\delta),L/N}\bigg).
\end{align}
\end{proof}

To obtain an upper bound for $\|v_{\Omega_{\delta,N,L}}\|_{\Leb^{\infty}(\Omega_{\delta,N,L})}$, we change the Dirichlet boundary conditions on $\partial C_L$ to Neumann boundary conditions. This increases the corresponding heat kernel, torsion function, and $\Leb^{\infty}$ norm respectively. By periodicity, we have that
\begin{equation}\label{e50}
\|v_{\Omega_{\delta,N,L}}\|_{\Leb^{\infty}(\Omega_{\delta,N,L})}\le \|\tilde{v}_{C_{L/N}-B(0;\delta)}\|_{\Leb^{\infty}(C_{L/N}-B(0;\delta))},
\end{equation}
where $\tilde{v}_{C_{L/N}-B(0;\delta)}$ is the torsion function with Neumann boundary conditions on $\partial C_{L/N}$, and Dirichlet boundary conditions on $\partial B(0;\delta)$.
Denote the spectrum of the corresponding Laplacian by $\{\mu_j:=\mu_{j,B(0;\delta),L/N},j=1,2,\dots\}$, and let $\{\varphi_j:=\varphi_{1,B(0;\delta),L/N},j=1,2,\dots\}$
denote a corresponding orthonormal basis of eigenfunctions. We denote by $\pi_{\delta,N/L}(x,y;t), x\in C_{L/N}-B(0;\delta), y \in C_{L/N}-B(0;\delta), t>0$ the corresponding heat kernel. Then
\begin{equation}\label{e51}
\pi_{\delta,N/L}(x,y;t)=\sum_{j=1}^{\infty}e^{-t\mu_j}\varphi_j(x)\varphi_j(y),
\end{equation}
and
\begin{align}\label{e52}
\tilde{v}&_{C_{L/N}-B(0;\delta)}(x)\nonumber \\ &
=\int_0^{\infty}dt\ \int_{C_{L/N}-B(0;\delta)}dy\ \pi_{\delta,N/L}(x,y;t)\bigg(\frac{\varphi_1(y)}{\|\varphi_1\|}+1-\frac{\varphi_1(y)}{\|\varphi_1\|}\bigg)\nonumber \\ &
=\frac{1}{\mu_1}\frac{\varphi_1(x)}{\|\varphi_1\|}+\int_0^{\infty}dt\ \int_{C_{L/N}-B(0;\delta)}dy\ \pi_{\delta,N/L}(x,y;t)\bigg(1-\frac{\varphi_1(y)}{\|\varphi_1\|}\bigg)\nonumber \\ & \le\frac{1}{\mu_1}+ \int_0^Tdt\ \int_{C_{L/N}-B(0;\delta)}dy\ \pi_{\delta,N/L}(x,y;t)\nonumber \\ &\hspace{1cm}+\int_T^{\infty}dt\ \int_{C_{L/N}-B(0;\delta)}dy\ \pi_{\delta,N/L}(x,y;t)\bigg(1-\frac{\varphi_1(y)}{\|\varphi_1\|}\bigg)\nonumber \\ & \le\frac{1}{\mu_1}+T+\int_T^{\infty}dt\ \int_{C_{L/N}-B(0;\delta)}dy\ \pi_{\delta,N/L}(x,y;t)\bigg(1-\frac{\varphi_1(y)}{\|\varphi_1\|}\bigg),
\end{align}
where $\|\varphi_1\|=\|\varphi_1\|_{\Leb^{\infty}(C_{L/N}-B(0;\delta))}$.
By \eqref{e51}, we have that the third term in the right-hand side of \eqref{e52} equals
\begin{equation}\label{e53}
\sum_{j=1}^{\infty}\mu_j^{-1}e^{-T\mu_j}\varphi_j(x)\int_{C_{L/N}-B(0;\delta)}dy\ \varphi_j(y)\bigg(1-\frac{\varphi_1(y)}{\|\varphi_1\|}\bigg).
\end{equation}
The term with $j=1$ in \eqref{e53} is bounded from above by
\begin{align*}%\label{e54}
\mu_1^{-1}\|\varphi_1\|\int_{C_{L/N}-B(0;\delta)}& \|\varphi_1\|\bigg(1-\frac{\varphi_1}{\|\varphi_1\|}\bigg)\nonumber \\ &
=\mu_1^{-1}\|\varphi_1\|\int_{C_{L/N}-B(0;\delta)}\big(\|\varphi_1\|-\varphi_1\big)\nonumber \\ &
\le \mu_1^{-1}\bigg(\|\varphi_1\|^2\bigg(\frac{L}{N}\bigg)^m-1\bigg),
\end{align*}
where we used the fact that $1=\int_{C_{L/N}-B(0;\delta)}\varphi_1^2\le \|\varphi_1\|\int_{C_{L/N}-B(0;\delta)}\varphi_1$.
It was shown on p.586, lines -3,-4, in \cite{vdBFNT} (with appropriate adjustment in notation) that
\begin{equation*}%\label{e55}
\|\varphi_1\|^2\le \bigg(\frac{N}{L}\bigg)^m\bigg(1-s\mu_1-\frac{mL^2}{3esN^2}\bigg)^{-1},s\ge 0,
\end{equation*}
provided the last term in the round brackets is non-negative. The optimal choice for $s$ gives that
\begin{equation*}%\label{e56}
\|\varphi_1\|^2\le \bigg(\frac{N}{L}\bigg)^m\bigg(1-\frac{(4m\mu_1)^{1/2}L}{(3e)^{1/2}N}\bigg)^{-1},\ \mu_1<\frac{3eN^2}{4mL^2}.
\end{equation*}
By further restricting the range for $\mu_1,$ we have that the first term with $j=1$ in \eqref{e53} is then bounded from above by
\begin{equation}\label{e57}
\mu_1^{-1}\frac{2L(m\mu_1/(3eN^2))^{1/2}}{1-2L(m\mu_1/(3eN^2))^{1/2}}\le \frac{(2m)^{1/2}L}{\mu_1^{1/2}N}, \ \mu_1\le\frac{3eN^2}{16mL^2}.
\end{equation}
The terms with $j\ge 2$ in \eqref{e53} give, by Cauchy-Schwarz for both the series in $j$, and the integral over $C_{L/N}-B(0;\delta)$, a contribution
\begin{align}\label{e58}
&\biggr\rvert\sum_{j=2}^{\infty}\mu_j^{-1}e^{-T\mu_j}\varphi_j(x)\int_{C_{L/N}-B(0;\delta)}\varphi_j\bigg(1-\frac{\varphi_1}{\|\varphi_1\|}\bigg)\biggr\rvert\nonumber \\ &
\le \mu_2^{-1}\sum_{j=2}^{\infty}e^{-T\mu_j}|\varphi_j(x)|\int_{C_{L/N}-B(0;\delta)}|\varphi_j|\nonumber \\ &
\le \mu_2^{-1}\bigg(\frac{L}{N}\bigg)^{m/2}\bigg(\sum_{j=2}^{\infty}e^{-T\mu_j}\bigg)^{1/2}\bigg(\sum_{j=2}^{\infty}e^{-T\mu_j}|\varphi_j(x)|^2\bigg)^{1/2}\nonumber \\ &
\le \mu_2^{-1}\bigg(\frac{L}{N}\bigg)^{m/2}\bigg(\sum_{j=2}^{\infty}e^{-T\mu_j}\bigg)^{1/2}\big(\pi_{\delta,N/L}(x,x;T)\big)^{1/2}.
\end{align}
To bound the first series in \eqref{e58}, we note that the $\mu_j$'s are bounded from below by the Neumann eigenvalues of the cube $C_{L/N}$. So choosing $T=(L/N)^2$ we get that
\begin{equation*}%\label{e59}
\bigg(\sum_{j=2}^{\infty}e^{-L^2\mu_j/N^2}\bigg)^{1/2}\le \bigg(1+\sum_{j=1}^{\infty}e^{-\pi^2j^2}\bigg)^{m/2}\le \bigg(\frac43\bigg)^{m/2}.
\end{equation*}
Similarly to the proof of Lemma 3.1 in \cite{vdBFNT}, we have that
\begin{align}\label{e60}
\big(\pi_{\delta,N/L}(x,x;L^2/N^2)\big)^{1/2}&\le \big(\pi_{0,N/L}(x,x;L^2/N^2)\big)^{1/2}\nonumber \\ &\le \bigg(\frac{N}{L}\bigg)^{m/2}\bigg(1+2\sum_{j=1}^{\infty}e^{-\pi^2j^2}\bigg)^{m/2}\nonumber \\ &\le
\bigg(\frac43\bigg)^{m/2}\bigg(\frac{N}{L}\bigg)^{m/2}.
\end{align}
Finally, $\mu_2\ge \frac{\pi^2N^2}{L^2}$, together with \eqref{e50}, \eqref{e52}, \eqref{e57}, \eqref{e58}, \eqref{e60}, and the choice $T=(L/N)^2$ gives that
\begin{align}\label{e61}
\|v_{\Omega_{\delta,N,L}}\|_{\Leb^{\infty}(\Omega_{\delta,N,L})}\le \mu_1^{-1}+\frac{(2m)^{1/2}L}{\mu_1^{1/2}N}+\bigg(\frac43\bigg)^{m}\frac{L^2}{N^2}, \, \mu_1\le\frac{3eN^2}{16mL^2}.
\end{align}
{\it Proof of Theorem \ref{the3}.} Let $1<\alpha<2$.
By \eqref{e49} and \eqref{e61}, we have that
\begin{align}\label{e62}
\lambda(\Omega_{\delta,N,L})\|v_{\Omega_{\delta,N,L}}\|_{\Leb^{\infty}(\Omega_{\delta,N,L})}\le & \bigg(\mu_1+32m\bigg(\frac54\bigg)^m\bigg(\frac{N}{L^2}+\frac{1}{N^{1/2}}\mu_1\bigg)\bigg) \nonumber \\ & \times \bigg(\mu_1^{-1}+\frac{(2m)^{1/2}L}{\mu_1^{1/2}N}+\bigg(\frac43\bigg)^{m}\frac{L^2}{N^2}\bigg),
\end{align}
provided
\begin{equation*}%\label{e63}
\frac{N}{L^2}\le \mu_1\le\frac{3eN^2}{16mL^2}.
\end{equation*}

First consider the planar case $m=2$. Recall Lemma 3.1 in \cite{vdBFNT}: for $\delta<L/(6N)$,
\begin{equation}\label{e64}
\frac{N^2}{100L^2}\bigg(\log \frac{L}{2\delta N}\bigg)^{-1}\le \mu_{1,B(0;\delta),L/N}\le \frac{8\pi N^2}{(4-\pi)L^2}\bigg(\log \frac{L}{2\delta N}\bigg)^{-1}.
\end{equation}
Let
\begin{equation}\label{e65}
\delta^*:=\delta^*(\alpha,N,L)=\frac{L}{2N}e^{-N^{2-\alpha}},
\end{equation}
where $1<\alpha<2$. Let $N_1\in \N$ be such that for all $N\ge N_1$, $\delta^*<L/(6N)$. We now use \eqref{e64} to see that there exists $C>1$ such that
\begin{equation}\label{e66}
C^{-1}\frac{N^{\alpha}}{L^2}\le \mu_{1,B(0;\delta^*),L/N}\le C\frac{N^{\alpha}}{L^2}.
\end{equation}
(In fact $C=\max\{100,8\pi/(4-\pi)\}$). We subsequently let $N_2\in \N$ be such that for all $N\ge N_2$,
\begin{equation*}%\label{e67}
\frac{N}{L^2}\le C^{-1}\frac{N^{\alpha}}{L^2}\le C\frac{N^{\alpha}}{L^2}\le \frac{3eN^2}{16mL^2}.
\end{equation*}
By \eqref{e62}, \eqref{e66}, and all $N\ge \max\{N_1,N_2\}$ we have that
 \begin{equation}\label{e68}
\lambda(\Omega_{\delta^*,N,L})\|v_{\Omega_{\delta^*,N,L}}\|_{\Leb^{\infty}(\Omega_{\delta^*,N,L})}\le 1+\mathcal{C}\big(N^{1-\alpha}+N^{(\alpha-2)/2}\big),
\end{equation}
where $\mathcal{C}$ depends on $C$ and on $m$ only. Finally, we let $N_3\in \N$ be such that for all $N\ge N_3$,
\begin{equation*}%\label{e69}
\mathcal{C}\big(N^{1-\alpha}+N^{(\alpha-2)/2}\big)<\epsilon.
\end{equation*}
We conclude that \eqref{e11} holds with $\Omega_{\epsilon}=\Omega_{\delta^*,N,L}$ with $\delta^*$ given by \eqref{e65}, and $N\ge \max\{N_1,N_2,N_3\}.$

Next consider the case $m=3,4,\dots$. We apply Lemma 3.2 in \cite{vdBFNT} to the case $K=B(0;\delta)$, and denote the Newtonian capacity of $K$ by $\textup{cap}(K)$. Then $\textup{cap}(B(0;\delta))=\kappa_m\delta^{m-2}$, where $\kappa_m$ is the Newtonian capacity of the ball with radius $1$ in $\R^m$. Then Lemma 3.2 gives that there exists $C\ge 1$ such that
\begin{equation}\label{e70}
C^{-1}\bigg(\frac{N}{L}\bigg)^m\delta^{m-2}\le\mu_{1,B(0;\delta,L/N)}\le C\bigg(\frac{N}{L}\bigg)^m\delta^{m-2},
\end{equation}
provided
\begin{equation}\label{e71}
\kappa_m\delta^{m-2}\le \frac{1}{16}(L/N)^{m-2}.
\end{equation}
We choose
\begin{equation}\label{e72}
\delta^*:=\delta^*(\alpha,N,L)=LN^{(\alpha-m)/(m-2)}.
\end{equation}
This gives inequality \eqref{e66} by \eqref{e70}.
The requirement \eqref{e71} holds for all $N\ge N_1$, where $N_1$ is the smallest natural number such that $N_1^{2-\alpha}\ge 16\kappa_m$.
The remainder of the proof follows the lines below \eqref{e66} with the appropriate adjustment of constants, and the choice of $\delta^*$ as in \eqref{e72}.
\hspace*{\fill }$\square $

We note that the choice $\alpha=\frac43$ in either \eqref{e65} or in \eqref{e72} gives, by \eqref{e68}, the decay rate
\begin{equation}\label{e73}
\lambda(\Omega_{\delta^*,N,L})\|v_{\Omega_{\delta^*,N,L}}\|_{\Leb^{\infty}(\Omega_{\delta^*,N,L})}\le 1+2\mathcal{C}N^{-1/3}.
\end{equation}

\section{Proof of Theorem \ref{the2}\label{sec3}}

In view of Payne's inequality \eqref{e9} it suffices to obtain an upper bound for $\|v_{\Omega}\|_{\Leb^{\infty}(\Omega)}\lambda(\Omega)$. We first observe, that by domain monotonicity of the torsion function, $v_{\Omega}$ is bounded by the torsion function for the (connected) set
bounded by the two parallel lines tangent to $\Omega$ at distance $w(\Omega)$. Hence
\begin{equation}\label{e33}
\|v_{\Omega}\|_{\Leb^{\infty}(\Omega)}\le \frac{w(\Omega)^2}{8}.
\end{equation}
In order to obtain an upper bound for $\lambda(\Omega)$, we introduce the following notation. For a planar, open, convex set, with finite measure, we let $z_1,z_2$ be two points on the boundary of $\Omega$ which realise the width. That is there are two parallel lines tangent to $\partial \Omega$, at $z_1$ and $z_2$ respectively, and at distance $w(\Omega)$. Let the $x$-axis be perpendicular to the vector $z_1z_2$, containing the point $\frac12(z_1+z_2)$. We consider the family of line segments parallel to the $x$-axis, obtained by intersection with $\Omega$, and let $l_1,l_2$ be two points on the boundary of $\Omega$ which realise the maximum length $L$ of this family. The quadrilateral with vertices, $z_1,z_2,l_1,l_2$ is contained in $\Omega$. This quadrilateral in turn contains a rectangle with side-lengths $h$, and $\big(1-\frac{h}{w(\Omega)}\big)L$ respectively, where $h\in [0,w(\Omega))
$ is arbitrary. Hence, by domain monotonicity of the Dirichlet eigenvalues, we have that
\begin{equation*}%\label{e34}
\lambda(\Omega)\le \pi^2h^{-2}+\pi^2\bigg(1-\frac{h}{w(\Omega)}\bigg)^{-2}L^{-2}.
\end{equation*}
Minimising the right-hand side above with respect to $h$ gives that
\begin{equation*}%\label{e35}
h= \frac{(w(\Omega)L^2)^{1/3}}{1+\big(\frac{L}{w(\Omega)}\big)^{2/3}}.
\end{equation*}
It follows that
\begin{equation}\label{e36}
\lambda(\Omega)\le\frac{\pi^2}{w(\Omega)^2}\bigg(1+\bigg(\frac{w(\Omega)}{L}\bigg)^{2/3}\bigg)^3.
\end{equation}
As $w(\Omega)\le L$ we obtain by \eqref{e36} that
\begin{align}\label{e37}
\lambda(\Omega)\le \frac{\pi^2}{w(\Omega)^2}\left(1+7\left(\frac{w(\Omega)}{L}\right)^{2/3}\right).
\end{align}
In order to complete the proof we need the following.
\begin{lemma}\label{lem1}
If $\Omega$ is an open, bounded, convex set in $\R^2$, and if $L$ is the length of the longest line segment in the closure of $\Omega$, perpendicular to $z_1z_2$, then
\begin{equation}\label{e38}
\textup{diam}(\Omega)\le 3L.
\end{equation}
\end{lemma}
\begin{proof}
Let $d_1,d_2\in \partial\Omega$ such that $|d_1-d_2|=\textup{diam}(\Omega)$. We denote the projections of $d_1,d_2$ onto the line through $z_1,z_2$ by $e_1,e_2$ respectively. Let $z$ be the intersection of the lines through $z_1,z_2$ and $d_1,d_2$ respectively. Then, by the maximality of $L$, we have that $|d_1-e_1|\le L, |d_2-e_2|\le L.$ Furthermore, by convexity, $|e_1-z|+|e_2-z|\le w(\Omega)$. Hence,
\begin{align*}%\label{e39}
|d_1-d_2|\le |d_1-e_1|+|e_1-z|+|d_2-e_2|+|e_2-z|\le 2L+w(\Omega)\le 3L.
\end{align*}
\end{proof}
By \eqref{e37}, we have that
\begin{equation*}%\label{e40}
\lambda(\Omega)\le\frac{\pi^2}{w(\Omega)^2}\left(1+7\cdot 3^{2/3}\left(\frac{w(\Omega)}{\textup{diam}(\Omega)}\right)^{2/3}\right).
\end{equation*}
This implies Theorem \ref{the2} by \eqref{e33}.
\hspace*{\fill }$\square $

\section{Proof of Theorem \ref{the1}\label{sec2}}

We denote by $d:M\times M\mapsto \R^+$ the geodesic distance associated to $(M,g)$. For $x\in M,\,R>0,$ $B(x;R)=\{y\in M:d(x,y)<R\}$. For a measurable set $A\subset M$ we denote by $|A|$ its Lebesgue measure. The Bishop-Gromov Theorem (see \cite{BC}) states that if $M$ is a complete, non-compact, $m$-dimensional, Riemannian manifold with non-negative Ricci curvature, then for $p\in M$,
the map $r\mapsto\frac{\vert B(p;r)\vert}{r^m}$ is
monotone decreasing. In particular
\begin{equation}\label{e18}
\frac{\vert B(p;r_2)\vert}{\vert B(p;r_1)\vert}\le
\left(\frac{r_2}{r_1}\right)^m,\ 0<r_1\le r_2.
\end{equation}
Corollary 3.1 and Theorem 4.1 in \cite{LY}, imply that if $M$ is complete with
non-negative Ricci curvature, then for any $D_2>2$ and $0<D_1<2$ there
exist constants $0<C_1\le C_2<\infty$ such that for all $x\in
M,\ y\in M,\ t>0$,
\begin{equation}\label{e19}
C_1\frac{e^{-d(x,y)^2/(2D_1t)}}{ (\vert B(x;t^{1/2})\vert \vert
B(y;t^{1/2})\vert )^{1/2}}\le p_M(x,y;t)\le
C_2\frac{e^{-d(x,y)^2/(2D_2t)}}{ (\vert B(x;t^{1/2})\vert \vert
B(y;t^{1/2})\vert )^{1/2}}.
\end{equation}
Finally, since by \eqref{e18} the measure of any geodesic ball with radius $r$ is bounded polynomially in $r$, the theorems of Grigor'yan in \cite{GB} imply stochastic completeness. That is, for all $x\in M,$ and all
$t>0$,
\begin{equation*}%\label{e20}
\int_Mdy\, p_M(x,y;t)=1.
\end{equation*}

\noindent {\it Proof of Theorem \ref{the1}.}
We choose $D_1=1,\, D_2=3$ in \eqref{e19}, and define the corresponding number $K=\max\{C_2,C_1^{-1}\}$. Then
\begin{equation}\label{e21}
K^{-1}e^{-d(x,y)^2/(2t)}\le (\vert B(x;t^{1/2})\vert \vert
B(y;t^{1/2})\vert )^{1/2}p_M(x,y;t)\le
Ke^{-d(x,y)^2/(6t)}.
\end{equation}

Let $q\in M$ be arbitrary, and let $R>0$ be such that $\Omega(q;R):=B(q;R)\cap \Omega\ne \emptyset$. The spectrum of the Dirichlet Laplacian acting in $L^2(\Omega(q;R))$ is discrete. Denote the bottom of this spectrum by $\lambda(\Omega(q;R))$. Then $\lambda(\Omega(q;R))\ge \lambda(\Omega)$. By the spectral theorem, monotonicity of Dirichlet heat kernels, and the Li-Yau bound \eqref{e21}, we have that
\begin{align}\label{e22}
p_{\Omega(q;R)}(x,x;t)&\le e^{-t\lambda(\Omega(q;R))/2}p_{\Omega(q;R)}(x,x;t/2)\nonumber \\ & \le e^{-t\lambda(\Omega(q;R))/2}p_M(x,x;t/2)\nonumber \\ &
\le Ke^{-t\lambda(\Omega(q;R))/2}|B(x;(t/2)^{1/2})|^{-1}.
\end{align}
By the semigroup property and the Cauchy-Schwarz inequality, for any open set $\Omega\subset M$, we have that
\begin{align}\label{e23}
p_{\Omega}(x,y;t)&=\int_{\Omega}dz\, p_{\Omega}(x,z;t/2)\,p_{\Omega}(z,y;t/2)\nonumber \\ &\le \left(\int_{\Omega}dz\, p_{\Omega}^2(x,z;t/2)\right)^{1/2}\left(\int_{\Omega}dz\, p_{\Omega}^2(z,y;t/2)\right)^{1/2}\nonumber \\ &=
\big(p_{\Omega}(x,x;t)\,p_{\Omega}(y,y;t)\big)^{1/2}.
\end{align}
We obtain by \eqref{e22}, \eqref{e23} (for $\Omega=\Omega(q;R)$), and $p_{\Omega(q;R)}(x,y;t)\le p_M(x,y;t),$ that
\begin{align}\label{e24}
p_{\Omega(q;R)}&(x,y;t)\le \big(p_{\Omega(q;R)}(x,x;t)\,p_{\Omega(q;R)}(y,y;t)\big)^{1/4}p_M(x,y;t)^{1/2}\nonumber \\ &\le K^{1/2}e^{-t\lambda(\Omega(q;R))/4}\big(|B(x;(t/2)^{1/2})||B(y;(t/2)^{1/2})|\big)^{-1/4}p_M^{1/2}(x,y;t).
\end{align}
By \eqref{e24} and \eqref{e21}, we have that
\begin{align}\label{e25}
p_{\Omega(q;R)}(x,y;t)\le & Ke^{-t\lambda(\Omega(q;R))/4}\big(|B(x;(t/2)^{1/2})||B(y;(t/2)^{1/2})|\big)^{-1/4} \nonumber \\ &\times\big(|B(x;t^{1/2})||
B(y;t^{1/2})|\big)^{-1/4}e^{-d(x,y)^2/(12t)}.
\end{align}
By the Li-Yau lower bound in \eqref{e21}, we can rewrite the right-hand side of \eqref{e25} to yield,
\begin{align}\label{e26}
p_{\Omega(q;R)}(x,y;t)&\le  K^2e^{-t\lambda(\Omega(q;R))/4}p_M(x,y;6t)\nonumber \\ & \times \frac{\big(|B(x;(6t)^{1/2})||B(y;(6t)^{1/2})|\big)^{1/2}}{\big(|B(x;(t/2)^{1/2})||B(y;(t/2)^{1/2})||B(x;t^{1/2})||
B(y;t^{1/2})|\big)^{1/4}}.
\end{align}
By Bishop-Gromov \eqref{e18}, we have that the volume quotients in the right-hand side of \eqref{e26} are bounded by $2^{3m/4}\cdot 3^{m/2}$ uniformly in $x$ and $y$.
Hence
\begin{equation*}%\label{e27}
p_{\Omega(q;R)}(x,y;t)\le 2^{3m/4}\cdot 3^{m/2}K^2e^{-t\lambda(\Omega(q;R))/4}p_M(x,y;6t).
\end{equation*}
Since manifolds with non-negative Ricci curvature are stochastically complete, we have that
\begin{align*}%\label{e28}
\int_{\Omega(q;R)}dy\,p_{\Omega(q;R)}(x,y;t)&\le 2^{3m/4}\cdot 3^{m/2}K^2e^{-t\lambda(\Omega(q;R))/4}\int_Mdy\,p_M(x,y;6t)\nonumber \\ &=2^{3m/4}\cdot 3^{m/2}K^2e^{-t\lambda(\Omega(q;R))/4}.
\end{align*}
Integrating the inequality above with respect to $t$ over $[0,\infty)$ yields,
\begin{equation*}%\label{e29}
v_{\Omega(q;R)}(x)\le 2^{(3m+8)/4}\cdot3^{m/2}K^2\lambda(\Omega(q;R))^{-1}\le 2^{(3m+8)/4}\cdot3^{m/2}K^2\lambda(\Omega)^{-1}.
\end{equation*}
Finally letting $R\rightarrow \infty$ in the left-hand side above yields the right-hand side of \eqref{e8}.

The proof of the left-hand side of \eqref{e8} is similar to the one in Theorem 5.3 in \cite{vdB} for Euclidean space.
We have that
\begin{equation}\label{e30}
    v_{\Omega(q;R)}(x)=\int_0^{\infty}dt\,\int_{\Omega(q;R)}dy\,p_{\Omega(q;R)}(x,y;t).
\end{equation}
We first observe that $|\Omega(q;R)|<\infty$, and so the spectrum of the Dirichlet Laplacian acting in $L^2(\Omega(q;R))$ is discrete and is denoted by
$\{\lambda_j(\Omega(q;R)), j\in \N\}$, with a corresponding orthonormal basis of eigenfunctions $\{\varphi_{j,\Omega(q;R)}, j\in \N\}$. These eigenfunctions are in $\Leb^{\infty}(\Omega(q;R))$.
Then, by \eqref{e30} and the eigenfunction expansion of the Dirichlet heat kernel for $\Omega(q;R)$, we have that
\begin{align}\label{e31}
v_{\Omega(q;R)}(x)&\ge\int_0^{\infty}dt\,\int_{\Omega(q;R)}dy\,p_{\Omega(q;R)}(x,y;t)\frac{\varphi_{1,\Omega(q;R)}(y)}{\|\varphi_{1,\Omega(q;R)}\|_{\Leb^{\infty}(\Omega(q;R))}}\nonumber \\ &
=\int_0^{\infty}dt\,e^{-t\lambda_1(\Omega(q;R))}\frac{\varphi_{1,\Omega(q;R)}(x)}{\|\varphi_{1,\Omega(q;R)}\|_{\Leb^{\infty}(\Omega(q;R))}}\nonumber \\ &=\lambda_1(\Omega(q;R))^{-1}\frac{\varphi_{1,\Omega(q;R)}(x)}{\|\varphi_{1,\Omega(q;R)}\|_{\Leb^{\infty}(\Omega(q;R))}}.
\end{align}
First taking the supremum over all $x\in \Omega(q;R)$ in the left-hand side of \eqref{e31}, and subsequently taking the supremum over all such $x$ in the right-hand side of \eqref{e31} gives
\begin{equation}\label{e32}
\|v_{\Omega(q;R)}\|_{\Leb^{\infty}(\Omega(q;R))}\ge \lambda(\Omega(q;R))^{-1}.
\end{equation}
Observe that the torsion function is monotone increasing in $R$. Taking the limit $R\rightarrow \infty$ in the left-hand side of \eqref{e32}, and subsequently in the right-hand side of \eqref{e32} completes the proof.
 \hspace*{\fill }$\square $

\end{document}